\documentclass{amsart}
\usepackage{amsmath,amsthm,amssymb}

\makeatletter
    
    \@addtoreset{equation}{section}
  \makeatother

\newtheorem{definition}{Definition}[section]
\newtheorem{proposition}[definition]{Proposition}
\newtheorem{theorem}[definition]{Theorem}
\newtheorem{lemma}[definition]{Lemma}

\newtheorem{remark}{Remark}[section]

\pagebreak

\title[Diffusion phenomenon for space-dependent damping]{
On diffusion phenomena for the linear wave equation with
space-dependent damping}
\author{Yuta WAKASUGI}
\email{y-wakasugi@cr.math.sci.osaka-u.ac.jp}
\address{Department of Mathematics, Graduate School of Science,
Osaka University, Osaka, Toyonaka, 560-0043, Japan}
\begin{document}
\begin{abstract}
In this paper, we prove the diffusion phenomenon for
the linear wave equation with space-dependent damping.
We prove that the asymptotic profile of the solution
is given by a solution of the corresponding heat equation in the
$L^2$-sense.
\end{abstract}

\maketitle
\section{Introduction}
In this paper, we consider the asymptotic behavior of
solutions to the wave equation with space-dependent damping:
\begin{equation}
\label{DW}
	\left\{\begin{array}{ll}
		u_{tt}-\Delta u+a(x)u_t=0,&(t,x)\in(0,\infty)\times\mathbf{R}^n,\\
		(u,u_t)(0,x)=(u_0,u_1)(x),&x\in\mathbf{R}^n.
	\end{array}\right.
\end{equation}
Here
$u=u(t,x)$
is real-valued unknown function and
$a(x)=\langle x\rangle^{-\alpha}:=(1+|x|^2)^{-\alpha/2}$
with
$0\le \alpha<1$.
For simplicity, we assume that
\begin{equation}
\label{eqini}
	(u_0,u_1)\in C_0^{\infty}(\mathbf{R}^n),\quad
	{\rm supp}\, (u_0,u_1)\subset \{ x\in\mathbf{R}^n \mid |x|\le L\}
\end{equation}
with some
$L>0$.
We also consider the initial value problem of the corresponding heat equation
started at the time
$\tau\ge 0$:
\begin{equation}
\label{H}
	\left\{\begin{array}{ll}
	a(x)v_t-\Delta v=0,&(t,x)\in (\tau,\infty)\times\mathbf{R}^n,\\
	v(\tau,x)=v_{\tau}(x),&x\in\mathbf{R}^n
	\end{array}\right.
\end{equation}
with initial data
$v_{\tau}(x)\in C_0^{\infty}(\mathbf{R}^n)$.

When
$a(x)=1$,
there are many literature on the asymptotic behavior of solutions to \eqref{DW}.
In this case, it is well known that the asymptotic profile of the solution of \eqref{DW}
is given by the solution of \eqref{H} with
the initial data
$v_0=u_0+u_1$
in several senses
(see \cite{HoOg04, Ik02, MaNi03, Na04, Ni03, YM00}
and see also \cite{CH03, IkNi03, RTY11}
for abstract setting).
On the other hand, Wirth \cite{Wi07} considered
the wave equation with time-dependent damping
$$
	u_{tt}-\Delta u+b(t)u_t=0.
$$
He proved that if the damping is effective,
that is, roughly speaking,
$tb(t)\rightarrow +\infty$
as
$t\rightarrow+\infty$
and
$b(t)^{-1}\notin L^1((0,\infty))$,
then the solution is asymptotically equivalent to that of
the corresponding heat equation
$$
	b(t)v_t-\Delta v=0
$$
(see also \cite{Ya06} for abstract setting).
We also mention that Ikehata, Todorova and Yordanov \cite{ITY13JDE}
recently proved the diffusion phenomenon for
strongly damped wave equations.

Recently, Nishiyama \cite{Ny13} proved the diffusion phenomenon
for abstract damped wave equations.
His result includes space-dependent damping
which does not decay near infinity.
Due to the authors knowledge, there are no results on the asymptotic profile
of solutions for decaying potential cases as \eqref{DW}.
The difficulty is that we cannot use the Fourier transform for \eqref{DW}
as the previous results.

We also refer the reader to
\cite{HKN04, HKN06, HKN07, HKN07JMAA, INZ06, Ka00, Ni06, Ni11AA}
for the asymptotic profile for semilinear problems.

Todorova and Yordanov \cite{TY09} obtained the following
$L^2$-estimates for \eqref{DW} and \eqref{H}:
\begin{align}
\label{L2estDW}
	\|u(t,\cdot)\|_{L^2}&\le C(1+t)^{-\frac{n-2\alpha}{2(2-\alpha)}+\varepsilon},\\
\label{L2estH}
	\|v(t,\cdot)\|_{L^2}&\le C(1+t)^{-\frac{n-2\alpha}{2(2-\alpha)}+\varepsilon}
\end{align}
with arbitrary small
$\varepsilon>0$
(see also \cite{ITY13} for the case $\alpha=1$).
It seems that the decay rate
$\frac{n-2\alpha}{2(2-\alpha)}$
is optimal.
Because, the function
$$
	G(t,x)=t^{-\frac{n-\alpha}{2-\alpha}}e^{-\frac{|x|^{2-\alpha}}{(2-\alpha)^2t}}
$$
formally satisfies the equation
$$
	|x|^{-\alpha}v_t-\Delta v=0
$$
and
$$
	\|G(t,\cdot)\|_{L^2}=Ct^{-\frac{n-2\alpha}{2(2-\alpha)}}
$$
with some constant
$C>0$.
Indeed,
\begin{align*}
	\|G(t,\cdot)\|_{L^2}^2
	&=t^{-\frac{2(n-\alpha)}{2-\alpha}}
		\int_{\mathbf{R}^n}e^{-\frac{2|x|^{2-\alpha}}{(2-\alpha)^2t}}dx\\
	&=Ct^{-\frac{2(n-\alpha)}{2-\alpha}}
		\int_0^{\infty}e^{-\frac{2r^{2-\alpha}}{(2-\alpha)^2t}}r^{n-1}dr\\
	&=Ct^{-\frac{2(n-\alpha)}{2-\alpha}+\frac{n-1}{2-\alpha}+\frac{1}{2-\alpha}}
		\int_0^{\infty}e^{-\frac{2r^{2-\alpha}}{(2-\alpha)^2t}}
			\left(\frac{r^{2-\alpha}}{t}\right)^{\frac{n-1}{2-\alpha}}
			d\left(\frac{r}{t^{1/(2-\alpha)}}\right)\\
	&=Ct^{-\frac{n-2\alpha}{2-\alpha}}.
\end{align*}

We denote the solution operator of \eqref{H} by
$E(t-\tau)$,
that is,
$v(t,x)=E(t-\tau)v_{\tau}(x)$
gives the solution of \eqref{H}.
It is known that
$E(t-\tau)$
is a $0$-th order pseudodifferential operator having the symbol
$$
	e(t-\tau,x,\xi)=e^{-\frac{|\xi|^2}{a(x)}(t-\tau)}+r_0(t-\tau,x,\xi)
$$
with a remainder term
$r_0$
(see Kumano-go \cite{Kutext}).
Our main result is the following.
\begin{theorem}\label{th1}
Let
$n\ge 1$
and let
$u$
be a solution of \eqref{DW}
with initial data
$(u_0,u_1)$
satisfying \eqref{eqini}.
Then we have
\begin{equation}
\label{DP}
	\left\| u(t,\cdot)-E(t)\left[u_0+\frac{1}{a(\cdot)}u_1\right](\cdot)\right\|_{L^2}
	=o(t^{-\frac{n-2\alpha}{2(2-\alpha)}})
\end{equation}
as
$t\rightarrow +\infty$.
\end{theorem}

\begin{remark}
Our proof needs the compactness of the support of the data.
However, this assumption may be removed by using
the energy concentration lemma (see Lemma \ref{lem3});
but we do not pursue that here.
\end{remark}

The crucial point of the proof of Theorem 1 is the following
weighted energy estimates for higher order derivatives
of solutions to \eqref{DW}.
Let
\begin{equation}
\label{psi}
	\psi(t,x)=A\frac{\langle x\rangle^{2-\alpha}}{1+t},\quad
	A:=\frac{1}{(2-\alpha)^2(2+\delta)}
\end{equation}
with small
$\delta>0$.
We also put
\begin{align*}
	I_0&=\int_{\mathbf{R}^n}e^{2\psi(0,x)}
		(u_0(x)^2+|\nabla u_0(x)|^2+|u_1(x)|^2)dx,\\
	I_1&=\int_{\mathbf{R}^n}e^{2\psi(0,x)}
		(u_{tt}(0,x)^2+|\nabla u_t(0,x)|^2)dx+I_0,\\
	I_2&=\int_{\mathbf{R}^n}e^{2\psi(0,x)}
		(u_{ttt}(0,x)^2+|\nabla u_{tt}(0,x)|^2)dx+I_1
\end{align*}
and by inductively
$$
	I_k=\int_{\mathbf{R}^n}e^{2\psi(0,x)}
		(|\partial_t^{k+1}u(0,x)|^2+|\nabla \partial_t^{k}u(0,x)|^2)dx+I_{k-1}.
$$
Then, we can obtain weighted energy estimates for any order of derivatives:
\begin{theorem}[weighted energy estimates for higher order derivatives]\label{th2}
For any small
$\varepsilon>0$,
there is some
$\delta>0$
such that the following estimates hold:
For any integer
$k\ge 0$,
there exists some constant
$C>0$
such that
for a solution
$u$
of \eqref{DW} with initial data satisfying \eqref{eqini}, we have
\begin{align}
\label{u1}
	(1+t)^{\frac{n-\alpha}{2-\alpha}+2k-\varepsilon}
		\int_{\mathbf{R}^n}e^{2\psi(t,x)}a(x)|\partial_t^ku(t,x)|^2dx&\le CI_k,\\
	\label{u2}
		(1+t)^{\frac{n-\alpha}{2-\alpha}+2k+1-\varepsilon}
			\int_{\mathbf{R}^n}e^{2\psi(t,x)}|\nabla \partial_t^ku(t,x)|^2dx
			&\le CI_k.
\end{align}
\end{theorem}
In particular, we use the following estimates for the proof of Theorem \ref{th1}.
\begin{lemma}\label{lem2}
For any small
$\varepsilon>0$,
there are some constants
$\delta>0$
and
$C>0$
such that the following estimates hold:
\begin{itemize}
\item[(i)]
	For a solution
	$u$
	of \eqref{DW}, we have
	\begin{align}
	\label{ut}
		(1+t)^{\frac{n-\alpha}{2-\alpha}+2-\varepsilon}
			\int_{\mathbf{R}^n}e^{2\psi(t,x)}a(x)u_t(t,x)^2dx&\le CI_1,\\
	\label{nablaut}
		(1+t)^{\frac{n-\alpha}{2-\alpha}+3-\varepsilon}
			\int_{\mathbf{R}^n}e^{2\psi(t,x)}|\nabla u_t(t,x)|^2dx
			&\le CI_1,\\
	\label{utt}
		(1+t)^{\frac{n-\alpha}{2-\alpha}+4-\varepsilon}
			\int_{\mathbf{R}^n}e^{2\psi(t,x)}a(x)u_{tt}(t,x)^2dx
			&\le CI_2,\\
	\label{uttt}
		(1+t)^{\frac{n-\alpha}{2-\alpha}+6-\varepsilon}
			\int_{\mathbf{R}^n}e^{2\psi(t,x)}a(x)u_{ttt}(t,x)^2dx&\le CI_3.
	\end{align}
\item[(ii)]
	For a solution
	$v$
	of \eqref{H}, we have
	\begin{align}
	\label{v}
		&\int_{\mathbf{R}^n}a(x)|v(t,x)|^2dx
			\le \int_{\mathbf{R}^n}a(x)|v_{\tau}(x)|^2dx,\\
	\label{vt}
		&\lefteqn{(1+t-\tau)^{\frac{n-\alpha}{2-\alpha}+2-\varepsilon}
			\int_{\mathbf{R}^n}a(x)|v_t(t,x)|^2dx}\\
\nonumber
		&\quad\le \int_{\mathbf{R}^n}a(x)^{-1}|\Delta v_{\tau}(x)|^2dx\\
\nonumber			
		&\quad\quad+C(1+\tau)^{\frac{n-\alpha}{2-\alpha}+1-\varepsilon}
			\int_{\mathbf{R}^n}e^{2\psi(\tau,x)}|\nabla v_{\tau}(x)|^2dx\\
\nonumber
		&\quad\quad+C(1+\tau)^{\frac{n-\alpha}{2-\alpha}-\varepsilon}
			\int_{\mathbf{R}^n}e^{2\psi(\tau,x)}a(x)v_{\tau}(x)^2dx.
\end{align}
\end{itemize}
\end{lemma}
We also use an energy concentration lemma.
\begin{lemma}[exponential decay outside parabolic regions]\label{lem3}
For any small
$\varepsilon>0$,
there is some
$\delta>0$
such that the following holds:
let
$$
	0<\rho<1-\alpha,\quad
	0<\mu<2A
$$
and
$$
	\Omega_{\rho}(t):=
	\{ x\in\mathbf{R}^n \mid \langle x\rangle^{2-\alpha}\ge (1+t)^{1+\rho} \}.
$$
We also assume that
$v$
is a solution of \eqref{H} with
$\tau=0$.
Then we have
\begin{equation}
\label{Expdecay}
	\int_{\Omega_{\rho}(t)}v(t,x)^2dx\le C(1+t)^{\frac{\alpha}{2-\alpha}}
	e^{-(2A-\mu)(1+t)^{\rho}}\int_{\mathbf{R}^n}e^{2\psi(0,x)}a(x)v_0(x)^2dx,
\end{equation}
where
$C>0$
is a constant depending on
$\delta, \rho$
and
$\mu$.
\end{lemma}

This paper is organized as follows.
In the next section,
we introduce a basic weighted energy method.
This method was originally developed by
Todorova and Yordanov \cite{TY01}
and refined by themselves \cite{TY09}
and Nishihara \cite{Ni10}
to fit for the space-dependent damping.
In Section 3, we prove Lemmas \ref{lem2} and \ref{lem3} and Thoerm \ref{th2}
by using the basic weighted energy estimates obtained in Section 2.
In the final section, we give a proof of the main theorem.

Finally, we explain some notation used in this paper.
First, we note that the letter
$C$
indicates the generic constant, which may change from line to line.
We also use the symbols $\lesssim$ and $\sim$.
The relation $f\lesssim g$ means $f\le Cg$ with some constant $C>0$
and $f\sim g$ means $f\lesssim g$ and $g\lesssim f$.
The bracket
$\langle \cdot \rangle$
is defined by
$\langle x\rangle:=\sqrt{1+|x|^2}$.
We denote the usual
$L^2$-norm by
$\|\cdot\|_{L^2}$,
that is,
$$
	\|f(\cdot)\|_{L^2}:=\left(\int_{\mathbf{R}^n}|f(x)|^2dx\right)^{1/2}.
$$

\section{Basic weighted energy estimates}
In this section, we give an estimate of a weighted
$L^2$-norm
$$
	\int_{\mathbf{R}^n}e^{2\psi(t,x)}a(x)u(t,x)^2dx
$$
and a weighted energy
$$
	\int_{\mathbf{R}^n}e^{2\psi(t,x)}(u_t(t,x)^2+|\nabla u(t,x)|^2)dx.
$$
The following estimate was essentially already obtained by
Nishihara \cite{Ni10}.
However, for the sake of completeness, we shall give a proof in this section.
\begin{proposition}[Basic weighted energy estimates]\label{propBE}
For any small
$\varepsilon>0$,
there is some
$\delta>0$
having the following property:
let
$u$
be a solution of \eqref{DW} with initial data
$(u_0,u_1)$
satisfying \eqref{eqini}.
Then we have
\begin{align*}
	&(1+t)^{\frac{n-\alpha}{2-\alpha}+1-\varepsilon}
		\int_{\mathbf{R}^n}e^{2\psi(t,x)}(u_t(t,x)^2+|\nabla u(t,x)|^2)dx\\
	&\quad+(1+t)^{\frac{n-\alpha}{2-\alpha}-\varepsilon}
		\int_{\mathbf{R}^n}e^{2\psi(t,x)}a(x)u(t,x)^2dx\\
	&\quad+\int_0^t\Big\{
		(1+\tau)^{\frac{n-\alpha}{2-\alpha}-\varepsilon}
			\int_{\mathbf{R}^n}e^{2\psi(\tau,x)}(u_t^2+|\nabla u|^2)(\tau,x)dx\\
	&\qquad+(1+\tau)^{\frac{n-\alpha}{2-\alpha}+1-\varepsilon}
			\int_{\mathbf{R}^n}e^{2\psi(\tau,x)}(-\psi_t(\tau,x))(u_t^2+|\nabla u|^2)(\tau,x)dx\\
	&\qquad+(t_0+t)^{\frac{n-\alpha}{2-\alpha}-\varepsilon}
			\int_{\mathbf{R}^n}e^{2\psi(\tau,x)}|\nabla\psi(\tau,x)|^2u(\tau,x)^2dx\\
	&\qquad+(1+\tau)^{\frac{n-\alpha}{2-\alpha}-1-\varepsilon}
		\int_{\mathbf{R}^n}e^{2\psi(\tau,x)}a(x)u(\tau,x)^2dx\\
	&\qquad+(1+\tau)^{\frac{n-\alpha}{2-\alpha}+1-\varepsilon}
		\int_{\mathbf{R}^n}e^{2\psi(\tau,x)}a(x)u_t(\tau,x)^2dx\Big\}d\tau\\
	&\le CI_0.
\end{align*}
\end{proposition}
\begin{proof}
Form \eqref{psi}, it is easy to see that
\begin{align}
\label{psit}
	-\psi_t&=\frac{1}{1+t}\psi,\\
\label{psix}
	\nabla\psi&=A\frac{(2-\alpha)\langle x\rangle^{-\alpha}x}{1+t},\\
\label{Deltapsi}
	\Delta\psi&=
		A(2-\alpha)(n-\alpha)\frac{\langle x\rangle^{-\alpha}}{1+t}
	+A(2-\alpha)\alpha\frac{\langle x\rangle^{-2-\alpha}}{1+t}\\
\nonumber
	&\ge \frac{n-\alpha}{(2-\alpha)(2+\delta)}\frac{a(x)}{1+t}\\
\nonumber
	&=:\left( \frac{n-\alpha}{2(2-\alpha)}-\delta_1\right)\frac{a(x)}{1+t}.
\end{align}
Here and after,
$\delta_i\ (i=1,2,\ldots)$
denote positive constants depending only on
$\delta$
such that
$$
	\delta_i\rightarrow 0^+ \quad \mbox{as} \quad\delta\rightarrow 0^+.
$$
We also have
\begin{align}
\label{psita}
	(-\psi_t)a(x)&=
		A\frac{\langle x\rangle^{2-2\alpha}}{(1+t)^{2}}\\
\nonumber
	&\ge\frac{1}{(2-\alpha)^2A}
		A^2(2-\alpha)^2\frac{\langle x\rangle^{-2\alpha}|x|^2}{(1+t)^{2}}\\
\nonumber
	&=(2+\delta)|\nabla\psi|^2.
\end{align}
By multiplying \eqref{DW} by
$e^{2\psi}u_t$,
it follows that
\begin{align}
\label{2mult1}
	\lefteqn{\frac{\partial}{\partial t}\left[ \frac{e^{2\psi}}{2}(u_t^2+|\nabla u|^2)\right]
		-\nabla\cdot(e^{2\psi}u_t\nabla u)}\\
\nonumber
	&\quad+e^{2\psi}\left(a(x)-\frac{|\nabla\psi|^2}{-\psi_t}-\psi_t\right)u_t^2
		+\underbrace{\frac{e^{2\psi}}{-\psi_t}|\psi_t\nabla u-u_t\nabla\psi|^2}_{T_1}
	=0.
\end{align}
Using the Schwarz inequality and \eqref{psita}, we can calculate
\begin{align*}
	T_1&=
	\frac{e^{2\psi}}{-\psi_t}(\psi_t^2|\nabla u|^2
		-2\psi_tu_t\nabla u\cdot\nabla\psi+u_t^2|\nabla\psi|^2)\\
	&\ge\frac{e^{2\psi}}{-\psi_t}
		\left(\frac{1}{5}\psi_t^2|\nabla u|^2-\frac{1}{4}u_t^2|\nabla\psi|^2\right)\\
	&\ge e^{2\psi}
		\left(\frac{1}{5}(-\psi_t)|\nabla u|^2-\frac{a(x)}{4(2+\delta)}u_t^2\right).
\end{align*}
This inequality and (\ref{psita}) lead
\begin{align}
\label{2mult2}
	\lefteqn{\frac{\partial}{\partial t}\left[ \frac{e^{2\psi}}{2}(u_t^2+|\nabla u|^2)\right]
		-\nabla\cdot(e^{2\psi}u_t\nabla u)}\\
\nonumber
	&\quad+e^{2\psi}
		\left\{\left(\frac{1}{4}a(x)
		-\psi_t\right)u_t^2+\frac{-\psi_t}{5}|\nabla u|^2\right\}
	\le 0.
\end{align}
Integrating over
$\mathbf{R}^n$,
we obtain
\begin{align}
\label{2mult3}
	\lefteqn{\frac{d}{dt}\int_{\mathbf{R}^n}\frac{e^{2\psi}}{2}(u_t^2+|\nabla u|^2)dx}\\
\nonumber
	&\quad+\int_{\mathbf{R}^n}e^{2\psi}
		\left\{\left(\frac{1}{4}a(x)
		-\psi_t\right)u_t^2+\frac{-\psi_t}{5}|\nabla u|^2\right\}dx
	\le0.
\end{align}

On the other hand,
we multiply \eqref{DW} by
$e^{2\psi}u$
and have
\begin{align}
\label{2mult4}
	\lefteqn{\frac{\partial}{\partial t}\left[e^{2\psi}\left(uu_t+\frac{a(x)}{2}u^2\right)\right]
		-\nabla\cdot (e^{2\psi}u\nabla u)}\\
\nonumber
	&\quad+e^{2\psi}\{|\nabla u|^2+(-\psi_t)a(x)u^2
	+\underbrace{2u\nabla\psi\cdot\nabla u}_{T_2}-2\psi_tuu_t-u_t^2\}
	=0.
\end{align}
We can rewrite the term $e^{2\psi}T_2$ as
\begin{align*}
	e^{2\psi}T_2&=
	4e^{2\psi}u\nabla\psi\cdot\nabla u-2e^{2\psi}u\nabla\psi\cdot\nabla u\\
	&=4e^{2\psi}u\nabla\psi\cdot\nabla u
		-\nabla\cdot(e^{2\psi}u^2\nabla\psi)
		+2e^{2\psi}u^2|\nabla\psi|^2+e^{2\psi}(\Delta\psi)u^2.
\end{align*}
By substituting this and using \eqref{Deltapsi}, it follows from \eqref{2mult4} that
\begin{align}
\label{2mult5}
	\lefteqn{\frac{\partial}{\partial t}\left[e^{2\psi}
		\left(uu_t+\frac{a(x)}{2}u^2\right)\right]
	-\nabla\cdot(e^{2\psi}(u\nabla u+u^2\nabla\psi))}\\
\nonumber
	&+e^{2\psi}\Big\{\underbrace{|\nabla u|^2+4u\nabla u\cdot\nabla\psi
		+((-\psi_t)a(x)+2|\nabla\psi|^2)u^2}_{T_3}\\
\nonumber
	&+\left(\frac{n-\alpha}{2-\alpha}-2\delta_1\right)\frac{a(x)}{2(1+t)}u^2
		-2\psi_tuu_t-u_t^2\Big\} \le 0.
\end{align}
By the Schwarz inequality, we can estimate the term $T_3$ as
\begin{align*}
	T_3&=|\nabla u|^2+4u\nabla u\cdot\nabla\psi\\
	&\quad+\left\{\left(1-\frac{\delta}{3}\right)(-\psi_t)a(x)
		+2|\nabla\psi|^2\right\}u^2+\frac{\delta}{3}(-\psi_t)a(x)u^2\\
	&\ge |\nabla u|^2+4u\nabla u\cdot\nabla\psi\\
	&\quad+\left(4+\frac{\delta}{3}-\frac{\delta^2}{3}\right)|\nabla\psi|^2u^2
		+\frac{\delta}{3}(-\psi_t)a(x)u^2\\
	&=\left(1-\frac{4}{4+\delta_2}\right)|\nabla u|^2+\delta_2|\nabla\psi|^2u^2\\
	&\quad+\left| \frac{2}{\sqrt{4+\delta_2}}\nabla u+\sqrt{4+\delta_2}u\nabla\psi\right|^2
			+\frac{\delta}{3}(-\psi_t)a(x)u^2\\
	&\ge\delta_3(|\nabla u|^2+|\nabla\psi|^2u^2)+\frac{\delta}{3}(-\psi_t)a(x)u^2.
\end{align*}
Substituting this into \eqref{2mult5}, we obtain
\begin{align*}
	\lefteqn{\frac{\partial}{\partial t}\left[e^{2\psi}
		\left(uu_t+\frac{a(x)}{2}u^2\right)\right]
		-\nabla\cdot(e^{2\psi}(u\nabla u+u^2\nabla\psi))}\\
\nonumber
	&\quad+e^{2\psi}\delta_3|\nabla u|^2\\
\nonumber
	&\quad+e^{2\psi}\left(\delta_3|\nabla\psi|^2+\frac{\delta}{3}(-\psi_t)a(x)
		+\left(\frac{n-\alpha}{2-\alpha}-2\delta_1\right)\frac{a(x)}{2(1+t)}\right)u^2\\
\nonumber
	&\quad+e^{2\psi}(-2\psi_tuu_t-u_t^2)\\
\nonumber
	&\le 0.
\end{align*}
Integration over
$\mathbf{R}^n$
leads
\begin{align}
\label{2mult6}
	\lefteqn{\frac{d}{dt}\int_{\mathbf{R}^n}e^{2\psi}
		\left(uu_t+\frac{a(x)}{2}u^2\right)dx}\\
\nonumber
	&\quad+\int_{\mathbf{R}^n}e^{2\psi}\delta_3|\nabla u|^2dx\\
\nonumber
	&\quad+\int_{\mathbf{R}^n}e^{2\psi}\left(\delta_3|\nabla\psi|^2+\frac{\delta}{3}(-\psi_t)a(x)
		+\left(\frac{n-\alpha}{2-\alpha}-2\delta_1\right)\frac{a(x)}{2(1+t)}\right)u^2dx\\
\nonumber
	&\quad+\int_{\mathbf{R}^n}e^{2\psi}(-2\psi_tuu_t-u_t^2)dx\\
\nonumber
	&\le 0.
\end{align}
By multiplying \eqref{2mult3} by
$(t_0+t)^{\alpha}$
with large parameter
$t_0$,
which determined later, we obtain
\begin{align}
\label{2mult7}
	\lefteqn{\frac{d}{dt}\left[
		(t_0+t)^{\alpha}\int_{\mathbf{R}^n}\frac{e^{2\psi}}{2}(u_t^2+|\nabla u|^2)dx
	\right]}\\
\nonumber
	&-\frac{\alpha}{2}(t_0+t)^{\alpha-1}\int_{\mathbf{R}^n}e^{2\psi}(u_t^2+|\nabla u|^2)dx\\
\nonumber
	&\quad+\int_{\mathbf{R}^n}e^{2\psi}
		\left\{\left(\frac{1}{4}
		+(t_0+t)^{\alpha}(-\psi_t)\right)u_t^2+(t_0+t)^{\alpha}\frac{-\psi_t}{5}|\nabla u|^2\right\}dx
	\le 0,
\end{align}
provided that
$t_0$
is sufficiently large, since
$a(x)\ge \langle t+L\rangle^{-\alpha}$
holds by the finite propagation speed property.
Let
$\nu$
be a positive small parameter depending on
$\delta$.
By \eqref{2mult7}$+\nu$\eqref{2mult6}, we have
\begin{align}
\label{2mult8}
	\lefteqn{\frac{d}{dt}\left[
		\int_{\mathbf{R}^n}e^{2\psi}\left\{
			\frac{(t_0+t)^{\alpha}}{2}(u_t^2+|\nabla u|^2)
			+\nu\left(uu_t+\frac{a(x)}{2}u^2\right)
		\right\}dx
	\right]}\\
\nonumber
	&\quad+\int_{\mathbf{R}^n}e^{2\psi}
		\left(\frac{1}{4}-\nu-\frac{\alpha}{2}(t_0+t)^{\alpha-1}
		+(t_0+t)^{\alpha}(-\psi_t)\right)u_t^2dx\\
\nonumber
	&\quad+\int_{\mathbf{R}^n}e^{2\psi}
		\left(\nu\delta_3-\frac{\alpha}{2}(t_0+t)^{\alpha-1}+(t_0+t)^{\alpha}\frac{-\psi_t}{5}
			\right)|\nabla u|^2dx\\
\nonumber
	&\quad+\nu\int_{\mathbf{R}^n}e^{2\psi}\left(\delta_3|\nabla\psi|^2+\frac{\delta}{3}(-\psi_t)a(x)
		+\left(\frac{n-\alpha}{2-\alpha}-2\delta_1\right)\frac{a(x)}{2(1+t)}\right)u^2dx\\
\nonumber
	&\quad+\nu\int_{\mathbf{R}^n}e^{2\psi}2(-\psi_t)uu_tdx\\
\nonumber
	&\le 0.
\end{align}
By the Schwarz inequality,
the last term of the left hand side of the above inequality can be estimated as
\begin{align*}
	|2(-\psi_t)uu_t|
	&\le \frac{\delta}{3}(-\psi_t)(t_0+t)^{-\alpha}u^2
	+\frac{3}{\delta}(-\psi_t)(t_0+t)^{\alpha}u_t^2\\
	&\le\frac{\delta}{3}(-\psi_t)a(x)u^2
	+\frac{3}{\delta}(-\psi_t)(t_0+t)^{\alpha}u_t^2.
\end{align*}
Therefore, it follows from \eqref{2mult8} that
\begin{align}
\label{2mult9}
	\lefteqn{\frac{d}{dt}\left[
		\int_{\mathbf{R}^n}e^{2\psi}\left\{
			\frac{(t_0+t)^{\alpha}}{2}(u_t^2+|\nabla u|^2)
			+\nu\left(uu_t+\frac{a(x)}{2}u^2\right)
		\right\}dx
	\right]}\\
\nonumber
	&\quad+\int_{\mathbf{R}^n}e^{2\psi}
		\left(\frac{1}{4}-\nu-\frac{\alpha}{2}(t_0+t)^{\alpha-1}
		+\left(1-\frac{3\nu}{\delta}\right)(t_0+t)^{\alpha}(-\psi_t)\right)u_t^2dx\\
\nonumber
	&\quad+\int_{\mathbf{R}^n}e^{2\psi}
		\left(\nu\delta_3-\frac{\alpha}{2}(t_0+t)^{\alpha-1}+(t_0+t)^{\alpha}\frac{-\psi_t}{5}
			\right)|\nabla u|^2dx\\
\nonumber
	&\quad+\nu\int_{\mathbf{R}^n}e^{2\psi}\left(\delta_3|\nabla\psi|^2
		+\left(\frac{n-\alpha}{2-\alpha}-2\delta_1\right)\frac{a(x)}{2(1+t)}\right)u^2dx\\
\nonumber
	&\le 0.
\end{align}
Now we choose the parameters
$\nu$
and
$t_0$
such that
\begin{align*}
	\frac{1}{4}-\nu-\frac{\alpha}{2}(t_0+t)^{\alpha-1}\ge c_0,\quad
	1-\frac{3\nu}{\delta}\ge c_0,\\
	\nu\delta_3-\frac{\alpha}{2}(t_0+t)^{\alpha-1}\ge c_0,\quad
	\frac{1}{5}\ge c_0
\end{align*}
hold for some constant
$c_0>0$.
This is possible because we first determine
$\nu$
sufficiently small depending on
$\delta$
and then we choose
$t_0$
sufficiently large depending on
$\nu$.
Consequently, we obtain
\begin{align}
\label{2mult10}
	\lefteqn{\frac{d}{dt}\left[
		\int_{\mathbf{R}^n}e^{2\psi}\left\{
			\frac{(t_0+t)^{\alpha}}{2}(u_t^2+|\nabla u|^2)
			+\nu\left(uu_t+\frac{a(x)}{2}u^2\right)
		\right\}dx
	\right]}\\
\nonumber
	&\quad+c_0\int_{\mathbf{R}^n}e^{2\psi}
		(1+(t_0+t)^{\alpha}(-\psi_t))(u_t^2+|\nabla u|^2)dx\\
\nonumber
	&\quad +\nu\int_{\mathbf{R}^n}e^{2\psi}\left(\delta_3|\nabla\psi|^2
		+\left(\frac{n-\alpha}{2-\alpha}-2\delta_1\right)\frac{a(x)}{2(1+t)}\right)u^2dx\\
\nonumber
	&\le 0.
\end{align}
We put
\begin{align*}
	\tilde{E}_1(t)&=\int_{\mathbf{R}^n}e^{2\psi}\left\{
			\frac{(t_0+t)^{\alpha}}{2}(u_t^2+|\nabla u|^2)
			+\nu\left(uu_t+\frac{a(x)}{2}u^2\right)
		\right\}dx,\\
	E_{1,\psi}(t)&=\int_{\mathbf{R}^n}e^{2\psi}
		(1+(t_0+t)^{\alpha}(-\psi_t))(u_t^2+|\nabla u|^2)dx\\
	\tilde{H}_1(t)&=\nu\int_{\mathbf{R}^n}e^{2\psi}\left(\delta_3|\nabla\psi|^2
		+\left(\frac{n-\alpha}{2-\alpha}-2\delta_1\right)\frac{a(x)}{2(1+t)}\right)u^2dx.
\end{align*}
Then we can rewrite \eqref{2mult10} as
\begin{equation}
\label{2energy1}
	\frac{d}{dt}\tilde{E}_1(t)+c_0E_{1,\psi}(t)+\tilde{H}_1(t)\le 0.
\end{equation}
Take arbitrary
$\varepsilon>0$
and we determine
$\delta$
so that
$\varepsilon=3\delta_1$.
By multiplying \eqref{2energy1} by
$(t_0+t)^{\frac{n-\alpha}{2-\alpha}-\varepsilon}$,
we have
\begin{align*}
	&\frac{d}{dt}[(t_0+t)^{\frac{n-\alpha}{2-\alpha}-\varepsilon}\tilde{E}_1(t)]
	-\left(\frac{n-\alpha}{2-\alpha}-\varepsilon\right)
		(t_0+t)^{\frac{n-\alpha}{2-\alpha}-1-\varepsilon}\tilde{E}_1(t)\\
	&\quad +c_0(t_0+t)^{\frac{n-\alpha}{2-\alpha}-\varepsilon}E_{1,\psi}(t)
	+(t_0+t)^{\frac{n-\alpha}{2-\alpha}-\varepsilon}\tilde{H}_1(t)\\
	&\le 0.
\end{align*}
Since
$$
	|\nu uu_t|\le
	\frac{\nu\delta_4}{2}a(x)u^2
	+\frac{\nu}{2\delta_4}(t_0+t)^{\alpha}u_t^2,
$$
we estimate
$$
	\tilde{E}_1(t)\le \int_{\mathbf{R}^n}e^{2\psi}
	\left( \left(1+\frac{\nu}{\delta_4}\right)\frac{(t_0+t)^{\alpha}}{2}u_t^2
		+\frac{(t_0+t)^{\alpha}}{2}|\nabla u|^2
		+\nu(1+\delta_4)\frac{a(x)}{2}u^2\right)dx.
$$
Choosing
$\delta_4$
sufficiently small
and then
$t_0$
sufficiently large so that
$$
	\left(\frac{n-\alpha}{2-\alpha}-2\delta_1\right)
	-\left(\frac{n-\alpha}{2-\alpha}-3\delta_1\right)(1+\delta_4)\ge c_1,
$$
$$
	c_0-\frac{1}{2}\left(1+\frac{\nu}{\delta_4}\right)(t_0+t)^{\alpha-1}\ge c_1
$$
with some
$c_1>0$,
we have
\begin{align*}
	&\frac{d}{dt}[(t_0+t)^{\frac{n-\alpha}{2-\alpha}-\varepsilon}\tilde{E}_1(t)]\\
	&\quad +c_1(t_0+t)^{\frac{n-\alpha}{2-\alpha}-\varepsilon}E_{1,\psi}(t)
	+(t_0+t)^{\frac{n-\alpha}{2-\alpha}-\varepsilon}H_1(t)\\
	&\le 0,
\end{align*}
where
$$
	H_1(t)=\nu\int_{\mathbf{R}^n}e^{2\psi}
	\left(\delta_3|\nabla \psi|^2+c_1\frac{a(x)}{2(1+t)}\right)u^2dx.
$$
Integrating over
$[0,t]$,
one can obtain
\begin{align*}
	(t_0+t)^{\frac{n-\alpha}{2-\alpha}-\varepsilon}\tilde{E}_1(t)
	+\int_0^t(t_0+\tau)^{\frac{n-\alpha}{2-\alpha}-\varepsilon}
		(c_1E_{1,\psi}(\tau)+H_1(\tau))d\tau
	\le \tilde{E}_1(0).
\end{align*}
We also put
$$
	E_1(t):=\int_{\mathbf{R}^n}e^{2\psi}
	\{ (t_0+t)^{\alpha}(u_t^2+|\nabla u|^2)+a(x)u^2\}dx.
$$
Then it is easy to see that
$\tilde{E}_1(t)\sim E_1(t)$
and
$E_1(0)\lesssim I_0$.
From this, we have
\begin{equation}
\label{2energy2}
	(t_0+t)^{\frac{n-\alpha}{2-\alpha}-\varepsilon}E_1(t)
	+\int_0^t(t_0+\tau)^{\frac{n-\alpha}{2-\alpha}-\varepsilon}
		(E_{1,\psi}(\tau)+H_1(\tau))d\tau
	\le CI_0.
\end{equation}

To reach the conclusion of the proposition,
we mutiply \eqref{2mult3} by
$(t_0+t)^{\frac{n-\alpha}{2-\alpha}+1-\varepsilon}$
and obtain
\begin{align*}
	\lefteqn{\frac{d}{dt}\left[(t_0+t)^{\frac{n-\alpha}{2-\alpha}+1-\varepsilon}
		\int_{\mathbf{R}^n}\frac{e^{2\psi}}{2}(u_t^2+|\nabla u|^2)dx\right] }\\
\nonumber
	&\quad-\left(\frac{n-\alpha}{2-\alpha}+1-\varepsilon\right)
		(t_0+t)^{\frac{n-\alpha}{2-\alpha}-\varepsilon}
		\int_{\mathbf{R}^n}\frac{e^{2\psi}}{2}(u_t^2+|\nabla u|^2)dx\\
\nonumber
	&\quad+(t_0+t)^{\frac{n-\alpha}{2-\alpha}+1-\varepsilon}
		\int_{\mathbf{R}^n}e^{2\psi}
		\left\{\left(\frac{1}{4}a(x)
		-\psi_t\right)u_t^2+\frac{-\psi_t}{5}|\nabla u|^2\right\}dx
	\le0.
\end{align*}
By integrating over
$[0,t]$, it holds that
\begin{align}
\label{2mult11}
	&(t_0+t)^{\frac{n-\alpha}{2-\alpha}+1-\varepsilon}
		\int_{\mathbf{R}^n}\frac{e^{2\psi}}{2}(u_t^2+|\nabla u|^2)dx\\
\nonumber
	&\quad-\left(\frac{n-\alpha}{2-\alpha}+1-\varepsilon\right)
		\int_0^t(t_0+\tau)^{\frac{n-\alpha}{2-\alpha}-\varepsilon}
		\int_{\mathbf{R}^n}\frac{e^{2\psi}}{2}(u_t^2+|\nabla u|^2)dxd\tau\\
\nonumber
	&\quad+\int_0^t(t_0+\tau)^{\frac{n-\alpha}{2-\alpha}+1-\varepsilon}
		\int_{\mathbf{R}^n}e^{2\psi}
		\left\{\left(\frac{1}{4}a(x)
		-\psi_t\right)u_t^2+\frac{-\psi_t}{5}|\nabla u|^2\right\}dxd \tau\\
\nonumber
	&\le CI_0.
\end{align}
Taking \eqref{2energy2}$+\eta$\eqref{2mult11}
with small parameter
$\eta>0$
satisfying
$$
	1-\frac{\eta}{2}\left(\frac{n-\alpha}{2-\alpha}+1-\varepsilon\right)>0,
$$
we can see that
\begin{align*}
	&(t_0+t)^{\frac{n-\alpha}{2-\alpha}+1-\varepsilon}
		\int_{\mathbf{R}^n}e^{2\psi(t,x)}(u_t(t,x)^2+|\nabla u(t,x)|^2)dx\\
	&\quad+(t_0+t)^{\frac{n-\alpha}{2-\alpha}-\varepsilon}
		\int_{\mathbf{R}^n}e^{2\psi(t,x)}a(x)u(t,x)^2dx\\
	&\quad+\int_0^t\Big\{
		(t_0+\tau)^{\frac{n-\alpha}{2-\alpha}-\varepsilon}
			\int_{\mathbf{R}^n}e^{2\psi(\tau,x)}(u_t^2+|\nabla u|^2)(\tau,x)dx\\
	&\qquad+(t_0+\tau)^{\frac{n-\alpha}{2-\alpha}+1-\varepsilon}
			\int_{\mathbf{R}^n}e^{2\psi(\tau,x)}(-\psi_t(\tau,x))(u_t^2+|\nabla u|^2)(\tau,x)dx\\
	&\qquad+(t_0+t)^{\frac{n-\alpha}{2-\alpha}-\varepsilon}
			\int_{\mathbf{R}^n}e^{2\psi(\tau,x)}|\nabla\psi(\tau,x)|^2u(\tau,x)^2dx\\
	&\qquad+(t_0+\tau)^{\frac{n-\alpha}{2-\alpha}-1-\varepsilon}
		\int_{\mathbf{R}^n}e^{2\psi(\tau,x)}a(x)u(\tau,x)^2dx\\
	&\qquad+(t_0+\tau)^{\frac{n-\alpha}{2-\alpha}+1-\varepsilon}
		\int_{\mathbf{R}^n}e^{2\psi(\tau,x)}a(x)u_t(\tau,x)^2dx\Big\}d\tau\\
	&\le CI_0.
\end{align*}
Finally, we note that
$(t_0+t)\sim (1+t)$
and obtain the conclusion.
\end{proof}

\section{Weighted energy estimates for higher order derivatives}
In this section, we give a proof of Lemmas \ref{lem2} and \ref{lem3}
and Theorem \ref{th2}.
We first prove \eqref{ut}.
Differentiating \eqref{DW} with respect to
$t$,
we obtain
\begin{equation}
\label{DWt}
	u_{ttt}-\Delta u_t+a(x)u_{tt}=0.
\end{equation}
We apply the weighted energy method again.
First, by Proposition \ref{propBE}, we have
\begin{equation}
\label{aut}
	\int_0^t(t_0+\tau)^{\frac{n-\alpha}{2-\alpha}+1-\varepsilon}
		\int_{\mathbf{R}^n}e^{2\psi(\tau,x)}a(x)u_t(\tau,x)^2dxd\tau
	\le CI_0.
\end{equation}
Multiplying \eqref{DWt}
by
$e^{2\psi}u_{tt}$
and
$e^{2\psi}u_t$,
and the same argument as the derivation of \eqref{2energy1},
we can obtain
\begin{equation}
\label{3energy1}
	\frac{d}{dt}\tilde{E}_2(t)+c_0E_{2,\psi}(t)+\tilde{H}_2(t)\le 0,
\end{equation}
where
\begin{align*}
	\tilde{E}_2(t)&=\int_{\mathbf{R}^n}e^{2\psi}\left\{
			\frac{(t_0+t)^{\alpha}}{2}(u_{tt}^2+|\nabla u_t|^2)
			+\nu\left(u_tu_{tt}+\frac{a(x)}{2}u_t^2\right)
		\right\}dx,\\
	E_{2,\psi}(t)&=\int_{\mathbf{R}^n}e^{2\psi}
		(1+(t_0+t)^{\alpha}(-\psi_t))(u_{tt}^2+|\nabla u_t|^2)dx,\\
	\tilde{H}_2(t)&=\nu\int_{\mathbf{R}^n}e^{2\psi}\left(\delta_3|\nabla\psi|^2
		+\left(\frac{n-\alpha}{2-\alpha}-2\delta_1\right)\frac{a(x)}{2(1+t)}\right)u_t^2dx.
\end{align*}
Multiplying \eqref{3energy1} by
$(t_0+t)^{\frac{n-\alpha}{2-\alpha}+2-\varepsilon}$
and retaking
$t_0$
larger,
we have
\begin{align*}
	&\frac{d}{dt}[(t_0+t)^{\frac{n-\alpha}{2-\alpha}+2-\varepsilon}\tilde{E}_2(t)]\\
	&\quad+c_1(t_0+t)^{\frac{n-\alpha}{2-\alpha}+2-\varepsilon}E_{2,\psi}(t)\\
	&\quad+\nu\delta_3(t_0+t)^{\frac{n-\alpha}{2-\alpha}+2-\varepsilon}
		\int_{\mathbf{R}^n}e^{2\psi}|\nabla \psi|^2u_t^2dx\\
	&\le C(t_0+t)^{\frac{n-\alpha}{2-\alpha}+1-\varepsilon}
		\int_{\mathbf{R}^n}e^{2\psi}a(x)u_t^2dx
\end{align*}
with some
$c_1>0$.
By \eqref{aut}, integrating over
$[0,t]$
and noting
$\tilde{E}_2(t)\sim E_2(t)$,
where
$$
	E_2(t)=\int_{\mathbf{R}^n}e^{2\psi}\left\{
			(t_0+t)^{\alpha}(u_{tt}^2+|\nabla u_t|^2)
			+a(x)u_t^2
		\right\}dx,
$$
it follows that
\begin{align}
\label{3energy2}
	&(t_0+t)^{\frac{n-\alpha}{2-\alpha}+2-\varepsilon}E_2(t)\\
\nonumber
	&\quad+\int_0^t
		(t_0+\tau)^{\frac{n-\alpha}{2-\alpha}+2-\varepsilon}E_{2,\psi}(\tau)d\tau\\
\nonumber
	&\quad+\int_0^t
		(t_0+\tau)^{\frac{n-\alpha}{2-\alpha}+2-\varepsilon}
		\int_{\mathbf{R}^n}e^{2\psi}|\nabla \psi|^2u_t^2dxd\tau\\
\nonumber
	&\le CI_1.
\end{align}
In particular, we can see that \eqref{ut} holds.
Furthermore, we obtain
\begin{equation}
\label{intutt}
	\int_0^t(t_0+\tau)^{\frac{n-\alpha}{2-\alpha}+2-\varepsilon}
		\int_{\mathbf{R}^n}e^{2\psi}(u_{tt}^2+|\nabla u_t|^2)dxd\tau
	\le CI_1.
\end{equation}
Using this, we can prove \eqref{nablaut}.
Indeed, by the same argument as proving \eqref{2mult3},
we have
\begin{align}
\label{3mult1}
	\lefteqn{\frac{d}{dt}\int_{\mathbf{R}^n}
		\frac{e^{2\psi}}{2}(u_{tt}^2+|\nabla u_t|^2)dx}\\
\nonumber
	&\quad+\int_{\mathbf{R}^n}e^{2\psi}
		\left\{\left(\frac{1}{4}a(x)
		-\psi_t\right)u_{tt}^2+\frac{-\psi_t}{5}|\nabla u_t|^2\right\}dx
	\le0.
\end{align}
Multiplying \eqref{3mult1} by
$(t_0+t)^{\frac{n-\alpha}{2-\alpha}+3-\varepsilon}$,
we obtain
\begin{align*}
	\lefteqn{\frac{d}{dt}\left[
		(t_0+t)^{\frac{n-\alpha}{2-\alpha}+3-\varepsilon}
		\int_{\mathbf{R}^n}
		\frac{e^{2\psi}}{2}(u_{tt}^2+|\nabla u_t|^2)dx
		\right]}\\
	&\quad+(t_0+t)^{\frac{n-\alpha}{2-\alpha}+3-\varepsilon}
	\int_{\mathbf{R}^n}e^{2\psi}
		\left\{ a(x)u_{tt}^2+(-\psi_t)(u_{tt}^2+|\nabla u_t|^2)
		\right\}dx\\
	&\le C(t_0+t)^{\frac{n-\alpha}{2-\alpha}+2-\varepsilon}
		\int_{\mathbf{R}^n}
		\frac{e^{2\psi}}{2}(u_{tt}^2+|\nabla u_t|^2)dx.
\end{align*}
Integration over the interval
$[0, t]$
and the estimate \eqref{intutt} imply
\begin{align}
\label{3energy3}
	\lefteqn{
		(t_0+t)^{\frac{n-\alpha}{2-\alpha}+3-\varepsilon}
		\int_{\mathbf{R}^n}
		e^{2\psi}(u_{tt}^2+|\nabla u_t|^2)dx}\\
\nonumber
	&\quad+\int_0^t(t_0+\tau)^{\frac{n-\alpha}{2-\alpha}+3-\varepsilon}
	\int_{\mathbf{R}^n}e^{2\psi}
		\left\{ a(x)u_{tt}^2+(-\psi_t)(u_{tt}^2+|\nabla u_t|^2)
		\right\}dxd\tau\\
\nonumber
	&\le CI_1.
\end{align}
In particular, we have \eqref{nablaut} and
\begin{equation}
\label{intutt2}
	\int_0^t(t_0+\tau)^{\frac{n-\alpha}{2-\alpha}+3-\varepsilon}
		\int_{\mathbf{R}^n}e^{2\psi}a(x)u_{tt}(\tau,x)^2dxd\tau
	\le CI_1,
\end{equation}
which will be used to obtain \eqref{utt} and \eqref{uttt}.

To prove \eqref{utt} and \eqref{uttt},
we differentiate \eqref{DWt} again and have
\begin{equation}
\label{DWtt}
	u_{tttt}-\Delta u_{tt}+a(x)u_{ttt}=0.
\end{equation}
Using \eqref{intutt2} instead of \eqref{aut} and
by the same argument as above,
we can prove instead of \eqref{3energy2} that
\begin{align}
\label{3energy4}
	&(t_0+t)^{\frac{n-\alpha}{2-\alpha}+4-\varepsilon}E_3(t)\\
\nonumber
	&\quad+\int_0^t
		(t_0+\tau)^{\frac{n-\alpha}{2-\alpha}+4-\varepsilon}E_{3,\psi}(\tau)d\tau\\
\nonumber
	&\quad+\int_0^t
		(t_0+t)^{\frac{n-\alpha}{2-\alpha}+4-\varepsilon}
		\int_{\mathbf{R}^n}e^{2\psi}|\nabla \psi|^2u_{tt}^2dxd\tau\\
\nonumber
	&\le CI_2,
\end{align}
where
\begin{align*}
	E_3(t)&=\int_{\mathbf{R}^n}e^{2\psi}\left\{
			(t_0+t)^{\alpha}(u_{ttt}^2+|\nabla u_{tt}|^2)
			+a(x)u_{tt}^2
		\right\}dx,\\
	E_{3,\psi}(t)&=\int_{\mathbf{R}^n}e^{2\psi}
		(1+(t_0+t)^{\alpha}(-\psi_t))(u_{ttt}^2+|\nabla u_{tt}|^2)dx.
\end{align*}
In particular, we obtain \eqref{utt}.
Moreover, by the same argument as deriving \eqref{3energy3},
one can obtain
\begin{align}
\label{3energy5}
	\lefteqn{
		(t_0+t)^{\frac{n-\alpha}{2-\alpha}+5-\varepsilon}
		\int_{\mathbf{R}^n}
		e^{2\psi}(u_{ttt}^2+|\nabla u_{tt}|^2)dx}\\
\nonumber
	&+\int_0^t(t_0+\tau)^{\frac{n-\alpha}{2-\alpha}+5-\varepsilon}
	\int_{\mathbf{R}^n}e^{2\psi}
		\left\{ a(x)u_{ttt}^2+(-\psi_t)(u_{ttt}^2+|\nabla u_{tt}|^2)
		\right\}dxd\tau\\
\nonumber
	&\le CI_3.
\end{align}
In particular, we have
\begin{equation}
\label{intuttt}
	\int_0^t(t_0+\tau)^{\frac{n-\alpha}{2-\alpha}+5-\varepsilon}
	\int_{\mathbf{R}^n}e^{2\psi}a(x)u_{ttt}^2dxd\tau\le CI_3.
\end{equation}
Using \eqref{intuttt} instead of \eqref{intutt2} again,
we can prove \eqref{uttt}.
Furthermore, we can continue the argument starting at \eqref{intutt2}
and obtaining \eqref{intuttt} as much as we want.
Therefore, we can obtain the conclusion of Theorem \ref{th2}.

Finally, we prove \eqref{v} and \eqref{vt}.
Multiplying \eqref{H} by
$v$,
we have
$$
	\frac{\partial}{\partial t}\left[ \frac{a(x)}{2}v^2 \right]
	-\nabla\cdot(v\nabla v)+|\nabla v|^2=0.
$$
Integrating over
$\mathbf{R}^n$,
one can obtain
$$
	\frac{1}{2}\frac{d}{dt}\int_{\mathbf{R}^n}a(x)v(t,x)^2dx
	+\int_{\mathbf{R}^n}|\nabla v(t,x)|^2dx=0.
$$
Thus, we have
$$
	\frac{1}{2}\int_{\mathbf{R}^n}a(x)v(t,x)^2dx
	+\int_{\tau}^t\int_{\mathbf{R}^n}|\nabla v(s,x)|^2dxds
	=\frac{1}{2}\int_{\mathbf{R}^n}a(x)v_{\tau}(x)^2dx,
$$
which implies \eqref{v}.
To prove \eqref{vt}, we apply a similar argument to \eqref{H} as in the previous section.
We first multiply \eqref{H} by
$e^{2\psi}v_t$
and have
\begin{align*}
	&\frac{\partial}{\partial t}\left[ \frac{e^{2\psi}}{2}|\nabla v|^2\right]
	-\nabla\cdot(e^{2\psi}v_t\nabla v)\\
	&\quad +e^{2\psi}\left(a(x)-\frac{|\nabla\psi|^2}{-\psi_t}\right)v_t^2
		+\underbrace{\frac{e^{2\psi}}{-\psi_t}|\psi_t\nabla v-v_t\nabla\psi|^2}_{T_1}=0.
\end{align*}
By \eqref{psita} and
$$
	T_1\ge e^{2\psi}\left(\frac{1}{5}(-\psi_t)|\nabla v|^2-\frac{a(x)}{4(2+\delta)}v_t^2\right),
$$
it follows that
\begin{align*}
	&\frac{\partial}{\partial t}\left[ \frac{e^{2\psi}}{2}|\nabla v|^2\right]
	-\nabla\cdot(e^{2\psi}v_t\nabla v)\\
\nonumber
	&\quad+e^{2\psi}\left(\frac{1}{4}a(x)v_t^2+\frac{1}{5}(-\psi_t)|\nabla v|^2\right)
	\le 0.
\end{align*}
Integrating over
$\mathbf{R}^n$,
we obtain
\begin{equation}
\label{3v1}
	\frac{d}{dt}\int_{\mathbf{R}^n}\frac{e^{2\psi}}{2}|\nabla v|^2dx
	+\int_{\mathbf{R}^n}
		e^{2\psi}\left(\frac{1}{4}a(x)v_t^2+\frac{1}{5}(-\psi_t)|\nabla v|^2\right)dx
	\le 0.
\end{equation}

On the other hand, by multiplying \eqref{H} by
$e^{2\psi}v$,
it follows that
\begin{align*}
	&\frac{\partial}{\partial t}\left[ e^{2\psi}\frac{a(x)}{2}v^2\right]
	-\nabla\cdot(e^{2\psi}v\nabla v)\\
	&\quad +e^{2\psi}\{ |\nabla v|^2+2v\nabla\psi\cdot\nabla v+(-\psi_t)a(x)v^2\}
	=0.
\end{align*}
By the same argument as the derivation of \eqref{2mult6},
we can see that
\begin{align}
\label{3v2}
	&\frac{d}{dt}\left[ \int_{\mathbf{R}^n}e^{2\psi}\frac{a(x)}{2}v^2dx\right]\\
\nonumber
	&\quad +\int_{\mathbf{R}^n}e^{2\psi}
		\left(\delta_3(|\nabla v|^2+|\nabla\psi|^2v^2)+\frac{\delta}{3}(-\psi_t)a(x)v^2
		+\left(\frac{n-\alpha}{2-\alpha}-2\delta_1\right)\frac{a(x)}{2(1+t)}v^2\right)dx\\
\nonumber
	&\le 0.
\end{align}
Taking  $\nu(1+t)^{\frac{n-\alpha}{2-\alpha}+1-\varepsilon}$
\eqref{3v1}$+(1+t)^{\frac{n-\alpha}{2-\alpha}-\varepsilon}$\eqref{3v2}
with small parameter
$\nu>0$,
we obtain
\begin{align*}
	&\frac{d}{dt}\left[
		\nu (1+t)^{\frac{n-\alpha}{2-\alpha}+1-\varepsilon}
		\int_{\mathbf{R}^n}\frac{e^{2\psi}}{2}|\nabla v|^2dx
		+(1+t)^{\frac{n-\alpha}{2-\alpha}-\varepsilon}
		\int_{\mathbf{R}^n}e^{2\psi}\frac{a(x)}{2}v^2dx\right]\\
	&\quad+\left(\delta_3-\frac{\nu}{2}\left(\frac{n-\alpha}{2-\alpha}+1-\varepsilon\right)\right)
		(1+t)^{\frac{n-\alpha}{2-\alpha}-\varepsilon}
			\int_{\mathbf{R}^n}e^{2\psi}|\nabla v|^2dx\\
	&\quad+\nu(1+t)^{\frac{n-\alpha}{2-\alpha}+1-\varepsilon}
		\int_{\mathbf{R}^n}e^{2\psi}
			\left(\frac{1}{4}a(x)v_t^2+\frac{(-\psi_t)}{5}|\nabla v|^2\right)dx\\
	&\quad+(1+t)^{\frac{n-\alpha}{2-\alpha}-\varepsilon}
		\int_{\mathbf{R}^n}e^{2\psi}
			\left( \delta_3|\nabla\psi|^2v^2+\frac{\delta}{3}(-\psi_t)a(x)v^2\right)dx\\
	&\quad+(\varepsilon-2\delta_1)(1+t)^{\frac{n-\alpha}{2-\alpha}-1-\varepsilon}
		\int_{\mathbf{R}^n}e^{2\psi}
			\frac{a(x)}{2}v^2dx\\
	&\le 0.
\end{align*}
We determine
$\delta$
so that
$\varepsilon=3\delta_1$
holds.
Then we take
$\nu$
sufficiently small.
Integrating the above inequality over
$[\tau, t]$,
we have
\begin{align}
	\label{3v3}
	&(1+t)^{\frac{n-\alpha}{2-\alpha}+1-\varepsilon}
		\int_{\mathbf{R}^n}\frac{e^{2\psi}}{2}|\nabla v|^2dx
		+(1+t)^{\frac{n-\alpha}{2-\alpha}-\varepsilon}
		\int_{\mathbf{R}^n}e^{2\psi}\frac{a(x)}{2}v^2dx\\
\nonumber
	&\quad+\int_{\tau}^t(1+s)^{\frac{n-\alpha}{2-\alpha}-\varepsilon}
			\int_{\mathbf{R}^n}e^{2\psi}|\nabla v|^2dxds\\
\nonumber
	&\quad+\int_{\tau}^t(1+s)^{\frac{n-\alpha}{2-\alpha}+1-\varepsilon}
		\int_{\mathbf{R}^n}e^{2\psi}(a(x)v_t^2+(-\psi_t)|\nabla v|^2)dxds\\
\nonumber
	&\quad+\int_{\tau}^t(1+s)^{\frac{n-\alpha}{2-\alpha}-\varepsilon}
		\int_{\mathbf{R}^n}e^{2\psi}
			\left( |\nabla\psi|^2v^2+(-\psi_t)a(x)v^2\right)dxds\\
\nonumber
	&\quad+\int_{\tau}^t(1+s)^{\frac{n-\alpha}{2-\alpha}-1-\varepsilon}
		\int_{\mathbf{R}^n}e^{2\psi}a(x)v^2dxds\\
\nonumber
	&\le C(1+\tau)^{\frac{n-\alpha}{2-\alpha}+1-\varepsilon}
		\int_{\mathbf{R}^n}e^{2\psi(\tau,x)}|\nabla v_{\tau}(x)|^2dx\\
\nonumber
	&\quad+C(1+\tau)^{\frac{n-\alpha}{2-\alpha}-\varepsilon}
		\int_{\mathbf{R}^n}e^{2\psi(\tau,x)}a(x)v_{\tau}(x)^2dx.
\end{align}
In particular, it follows that
\begin{align}
\label{3v4}
	&\int_{\tau}^t(1+s)^{\frac{n-\alpha}{2-\alpha}+1-\varepsilon}
		\int_{\mathbf{R}^n}e^{2\psi(s,x)}a(x)v_t(s,x)^2dxds\\
\nonumber
	&\quad\le C(1+\tau)^{\frac{n-\alpha}{2-\alpha}+1-\varepsilon}
		\int_{\mathbf{R}^n}e^{2\psi(\tau,x)}|\nabla v_{\tau}(x)|^2dx\\
\nonumber
	&\quad+C(1+\tau)^{\frac{n-\alpha}{2-\alpha}-\varepsilon}
		\int_{\mathbf{R}^n}e^{2\psi(\tau,x)}a(x)v_{\tau}(x)^2dx.
\end{align}
Using this estimate, we prove \eqref{vt}.
We differentiate \eqref{H} with respect to
$t$
and have
$$
	a(x)v_{tt}-\Delta v_t=0.
$$
Multiplying this by
$v_t$,
we obtain
$$
	\frac{\partial}{\partial t}\left[ \frac{a(x)}{2}v_t^2 \right]
	-\nabla\cdot(v_t\nabla v_t)+|\nabla v_t|^2=0.
$$
Integrating over
$\mathbf{R}^n$,
one can see that
$$
	\frac{d}{dt}\int_{\mathbf{R}^n}a(x)v_t(t,x)^2dx\le 0.
$$
Moreover, taking account into \eqref{3v4}, multiplying this by
$(1+t-\tau)^{\frac{n-\alpha}{2-\alpha}+2-\varepsilon}$
with
$0\le \tau\le t$,
we have
\begin{align*}
	\frac{d}{dt}\left[
		(1+t-\tau)^{\frac{n-\alpha}{2-\alpha}+2-\varepsilon}
		\int_{\mathbf{R}^n}a(x)v_t(t,x)^2dx\right]
	&\le C(1+t-\tau)^{\frac{n-\alpha}{2-\alpha}+1-\varepsilon}
		\int_{\mathbf{R}^n}a(x)v_t(t,x)^2dx\\
	&\le C(1+t)^{\frac{n-\alpha}{2-\alpha}+1-\varepsilon}
		\int_{\mathbf{R}^n}a(x)v_t(t,x)^2dx.
\end{align*}
Integrating over
$[\tau, t]$,
and using \eqref{3v4},
we can conclude that
\begin{align*}
	&(1+t-\tau)^{\frac{n-\alpha}{2-\alpha}+2-\varepsilon}
		\int_{\mathbf{R}^n}a(x)v_t(t,x)^2dx\\
	&\le \int_{\mathbf{R}^n}a(x)v_t(\tau,x)^2dx\\
	&\quad +C\int_{\tau}^t(1+s)^{\frac{n-\alpha}{2-\alpha}+1-\varepsilon}
		\int_{\mathbf{R}^n}a(x)v_t(s,x)^2dxds\\
	&\le \int_{\mathbf{R}^n}a(x)^{-1}|\Delta v_\tau(x)|^2dx\\
	&\quad+ C(1+\tau)^{\frac{n-\alpha}{2-\alpha}+1-\varepsilon}
		\int_{\mathbf{R}^n}e^{2\psi(\tau,x)}|\nabla v_{\tau}(x)|^2dx\\
	&\quad+C(1+\tau)^{\frac{n-\alpha}{2-\alpha}-\varepsilon}
		\int_{\mathbf{R}^n}e^{2\psi(\tau,x)}a(x)v_{\tau}(x)^2dx.
\end{align*}
Here we note that
$$
	a(x)v_t(\tau,x)=\Delta v_{\tau}(x),
$$
since
$v$
satisfies \eqref{H}.
Thus, we obtain \eqref{vt}.

\begin{proof}[Proof of Lemma 1.4]
By \eqref{3v2}, we have
$$
	\frac{d}{dt}\left[\int_{\mathbf{R}^n}e^{2\psi}\frac{a(x)}{2}v^2dx\right]\le 0.
$$
This shows
\begin{equation}
\label{3v5}
	\int_{\mathbf{R}^n}e^{2\psi(t,x)}\frac{a(x)}{2}v(t,x)^2dx
	\le\int_{\mathbf{R}^n}e^{2\psi(0,x)}\frac{a(x)}{2}v(0,x)^2dx.
\end{equation}
Let
$0<\rho<1-\alpha$,
$0<\mu<2A$
and
$$
	\Omega_{\rho}(t):=
	\{ x\in\mathbf{R}^n \mid \langle x\rangle^{2-\alpha}\ge (1+t)^{1+\rho} \}.
$$
A simple calculation implies
$$
	e^{2\psi}a(x)\ge c(1+t)^{-\frac{\alpha}{2-\alpha}}
		e^{(2A-\mu)\frac{\langle x\rangle^{2-\alpha}}{1+t}}.
$$
By noting that
$$
	\frac{\langle x\rangle^{2-\alpha}}{1+t}\ge (1+t)^{\rho}
$$
on
$\Omega_{\rho}(t)$
and \eqref{3v5}, it follows that
\begin{align*}
	\lefteqn{(1+t)^{-\frac{\alpha}{2-\alpha}}\int_{\Omega_{\rho}(t)}
		e^{(2A-\mu)(1+t)^{\rho}}v(t,x)^2dx}\\
	&\le C(1+t)^{-\frac{\alpha}{2-\alpha}}\int_{\Omega_{\rho}(t)}
		e^{(2A-\mu)\frac{\langle x\rangle^{2-\alpha}}{1+t}}v(t,x)^2dx\\
	&\le C\int_{\mathbf{R}^n}e^{2\psi(t,x)}a(x)v(t,x)^2dx\\
	&\le C\int_{\mathbf{R}^n}e^{2\psi(0,x)}a(x)v(0,x)^2dx.
\end{align*}
Thus, we obtain
$$
	\int_{\Omega_{\rho}(t)}v(t,x)^2dx\le
	C(1+t)^{\frac{\alpha}{2-\alpha}}
		e^{-(2A-\mu)(1+t)^{\rho}}\int_{\mathbf{R}^n}e^{2\psi(0,x)}a(x)v(0,x)^2dx.
$$
This proves Lemma \ref{lem3}.

\end{proof}

\section{Proof of the main theorem}
We can write
\begin{equation}
\label{Duhamel}
	u(t,x)=E(t)u_0(x)-\int_0^tE(t-\tau)[a(\cdot)^{-1}u_{tt}(\tau,\cdot)](x)d\tau.
\end{equation}
By the integration by parts, we have
\begin{align*}
	-\int_0^tE(t-\tau)[a(\cdot)^{-1}u_{tt}(\tau,\cdot)]d\tau
	&=-\int_{t/2}^tE(t-\tau)[a(\cdot)^{-1}u_{tt}(\tau,\cdot)]d\tau\\
	&\quad-E(t/2)[a^{-1}u_{t}(t/2)]+E(t)[a^{-1}u_1]\\
	&\quad -\int_0^{t/2}\frac{\partial E}{\partial t}(t-\tau)[a^{-1}u_t(\tau)]d\tau,
\end{align*}
where
$\frac{\partial E}{\partial t}(t-\tau)$
is the pseudodifferential operator with the symbol
$\frac{\partial e}{\partial t}(t-\tau,x,\xi)$,
that is,
$\frac{\partial E}{\partial t}(t-\tau)v_{\tau}$
denotes the derivative of
$E(t-\tau)v_{\tau}$
with respect to
$t$.
Therefore, we obtain
\begin{align}
\label{u-E}
	u(t)-E(t)[u_0+a^{-1}u_1]&=
		-\int_{t/2}^tE(t-\tau)[a(\cdot)^{-1}u_{tt}(\tau,\cdot)]d\tau\\
\nonumber
		&\quad-E(t/2)[a^{-1}u_{t}(t/2)]\\
\nonumber
		&\quad-\int_0^{t/2}\frac{\partial E}{\partial t}(t-\tau)[a^{-1}u_t(\tau)]d\tau
\end{align}
and it suffices to prove that
the each term of the right hand side are
$o(t^{-\frac{n-2\alpha}{2(2-\alpha)}})$
in the
$L^2$-sense.
First, by the finite propagation speed property,
we have
$u(t,x)=\chi(t,x)u(t,x)$
with the characteristic function
$\chi(t,x)$
of the region
$\{ (t,x)\in (0,\infty)\times\mathbf{R}^n \mid |x|<t+L \}$.
Moreover, by Lemma \ref{lem3},
the $L^2$-norm of
$(1-\chi(t,x))E(t)[u_0+a^{-1}u_1]$
decays exponentially.
Thus, by multiplying \eqref{u-E} by
$\chi(t,x)$,
it suffices to estimate the terms
$$
	K_1:=\chi(t,x)\int_{t/2}^tE(t-\tau)[a(\cdot)^{-1}u_{tt}(\tau,\cdot)]d\tau,
	\quad K_2:=\chi(t,x)E(t/2)[a^{-1}u_{t}(t/2)]
$$
and
$$
	K_3:=\chi(t,x)\int_0^{t/2}\frac{\partial E}{\partial t}(t-\tau)[a^{-1}u_t(\tau)]d\tau.
$$

We first estimate
$K_1$.
By \eqref{v} and \eqref{utt}, we have
\begin{align*}
	\|K_1\|_{L^2}&=
		\left\|\chi(t)\int_{t/2}^tE(t-\tau)[a^{-1}u_{tt}(\tau)]d\tau\right\|_{L^2}\\
	&=\left\|\chi(t)\frac{1}{\sqrt{a}}
		\int_{t/2}^t\sqrt{a}E(t-\tau)[a^{-1}u_{tt}(\tau)]d\tau\right\|_{L^2}\\
	&\lesssim (1+t)^{\alpha/2}\int_{t/2}^t
		\|\sqrt{a}E(t-\tau)[a^{-1}u_{tt}(\tau)]\|_{L^2}d\tau\\
	&\le (1+t)^{\alpha/2}\int_{t/2}^t
		\|\sqrt{a}a^{-1}u_{tt}(\tau)\|_{L^2}d\tau\\
	&\lesssim (1+t)^{3\alpha/2}\int_{t/2}^t
		\|\sqrt{a}u_{tt}(\tau)\|_{L^2}d\tau\\
	&\lesssim (1+t)^{3\alpha/2-\frac{n-\alpha}{2(2-\alpha)}-1-\varepsilon/2}
	=o(t^{-\frac{n-2\alpha}{2(2-\alpha)}}),
\end{align*}
provided that
$\varepsilon>0$
is taken sufficiently small.
Because, it is true that
\begin{equation}
\label{exponent}
	\frac{3}{2}\alpha-\frac{n-\alpha}{2(2-\alpha)}-1<-\frac{n-2\alpha}{2(2-\alpha)}
\end{equation}
holds if
$0\le\alpha<1$.

We can estimate
$K_2$
by a similar way.
Using \eqref{v}, \eqref{ut} and \eqref{exponent}, we obtain
\begin{align*}
	\|K_2\|_{L^2}&=
		\left\|\chi(t)\frac{1}{\sqrt{a}}\sqrt{a}E(t/2)[a^{-1}u_t(t/2)]\right\|_{L^2}\\
	&\lesssim (1+t)^{\alpha/2}\left\| \sqrt{a}E(t/2)[a^{-1}u_t(t/2)]\right\|_{L^2}\\
	&\lesssim (1+t)^{\alpha/2}\left\| \sqrt{a}a^{-1}u_t(t/2)\right\|_{L^2}\\
	&\lesssim (1+t)^{3\alpha/2}\left\| \sqrt{a}u_t(t/2)\right\|_{L^2}\\
	&\lesssim (1+t)^{3\alpha/2-\frac{n-\alpha}{2(2-\alpha)}-1+\varepsilon/2}
	=o(t^{-\frac{n-2\alpha}{2(2-\alpha)}}).
\end{align*}

Finally, we estimate
$K_3$.
By \eqref{vt}, we have
\begin{align*}
	\|K_3\|_{L^2}&=
		\left\|\chi(t)\int_0^{t/2}\frac{\partial E}{\partial t}(t-\tau)
			[a^{-1}u_t(\tau)]d\tau\right\|_{L^2}\\
	&=\left\|\chi(t)\frac{1}{\sqrt{a}}\int_0^{t/2}\sqrt{a}\frac{\partial E}{\partial t}(t-\tau)
			[a^{-1}u_t(\tau)]d\tau\right\|_{L^2}\\
	&\lesssim (1+t)^{\alpha/2}\int_0^{t/2}
		\left\| \sqrt{a}\frac{\partial E}{\partial t}(t-\tau)[a^{-1}u_t(\tau)]\right\|_{L^2}d\tau\\
	&\lesssim (1+t)^{\alpha/2}\int_0^{t/2}
		(1+t-\tau)^{-\frac{n-\alpha}{2(2-\alpha)}-1+\varepsilon/2}
		(J_1+J_2+J_3)d\tau\\
	&\lesssim (1+t)^{\alpha/2-\frac{n-\alpha}{2(2-\alpha)}-1+\varepsilon/2}
		\int_0^{t/2}(J_1+J_2+J_3)d\tau,
\end{align*}
where
\begin{align*}
	J_1^2&=\int_{\mathbf{R}^n}a(x)^{-1}
		|\Delta (a(x)^{-1}u_t(\tau,x))|^2dx,\\
	J_2^2&=(1+\tau)^{\frac{n-\alpha}{2-\alpha}+1-\varepsilon}
		\int_{\mathbf{R}^n}e^{2\psi(\tau,x)}|\nabla (a(x)^{-1}u_t(\tau,x))|^2dx
\end{align*}
and
$$
	J_3^2=(1+\tau)^{\frac{n-\alpha}{2-\alpha}-\varepsilon}
		\int_{\mathbf{R}^n}e^{2\psi(\tau,x)}a(x)|a(x)^{-1}u_t(\tau,x)|^2dx.
$$
By noting
\begin{equation}
\label{exponent2}
	\alpha-\frac{n-\alpha}{2(2-\alpha)}-\frac{1}{2}<-\frac{n-2\alpha}{2(2-\alpha)}
\end{equation}
if
$0\le \alpha<1$,
it is only necessary to prove
\begin{equation}
\label{jk}
	\int_0^{t/2}J_kd\tau\le C(1+t)^{\frac{\alpha+1}{2}}
\end{equation}
for
$k=1,2,3$
with some constant
$C>0$.

Now we prove \eqref{jk}.
We first estimate
$J_3$.
By \eqref{ut} and the finite propagation speed property again,
we can estimate
\begin{align*}
	J_3^2&\lesssim (1+\tau)^{\frac{n-\alpha}{2-\alpha}-\varepsilon}(1+\tau)^{2\alpha}
		\int_{\mathbf{R}^n}e^{2\psi(\tau,x)}a(x)u_t(\tau,x)^2dx\\
	&\lesssim (1+\tau)^{2\alpha-2}.
\end{align*}
By a simple calculation, we can see that
$$
	\int_0^{t/2}J_3d\tau\lesssim
	\int_0^{t/2}(1+\tau)^{\alpha-1}d\tau
	\lesssim (1+t)^{\frac{\alpha+1}{2}}.
$$
Next, we esitmate
$J_2$.
Noting
$$
	\nabla(a^{-1}u_t)=\nabla (a^{-1})u_t+a^{-1}\nabla u_t
$$
and
$|\nabla(a^{-1})|\lesssim \langle x\rangle^{\alpha-1}$,
we have
$$
	J_2^2\lesssim (1+\tau)^{\frac{n-\alpha}{2-\alpha}+1-\varepsilon}
		\int_{\mathbf{R}^n}e^{2\psi(\tau,x)}
		(|\langle x\rangle^{\alpha-1}u_t(\tau,x)|^2
		+|\langle x\rangle^{\alpha}\nabla u_t(\tau,x)|^2)dx.
$$
By \eqref{ut} and \eqref{utt}, we obtain
\begin{align*}
	\int_{\mathbf{R}^n}e^{2\psi(\tau,x)}|\langle x\rangle^{\alpha-1}u_t(\tau,x)|^2dx
	&\lesssim (1+\tau)^{\alpha}
		\int_{\mathbf{R}^n}e^{2\psi(\tau,x)}a(x)u_t(\tau,x)^2dx\\
		&\lesssim (1+\tau)^{\alpha-\frac{n-\alpha}{2-\alpha}-2+\varepsilon}
\end{align*}
and
\begin{align*}
	\int_{\mathbf{R}^n}e^{2\psi(\tau,x)}|\langle x\rangle^{\alpha}\nabla u_t(\tau,x)|^2dx
	&\lesssim (1+\tau)^{2\alpha}
	\int_{\mathbf{R}^n}e^{2\psi(\tau,x)}|\nabla u_t(\tau,x)|^2dx\\
	&\lesssim (1+\tau)^{2\alpha-\frac{n-\alpha}{2-\alpha}-3+\varepsilon}.
\end{align*}
Therefore, it holds that
$$
	\int_0^{t/2}J_2d\tau\lesssim (1+t)^{\frac{\alpha+1}{2}}.
$$

Finally, we estimate
$J_1$.
Noting
$$
	\Delta (a^{-1}u_t)=\Delta(a^{-1})u_t+2\nabla(a^{-1})\cdot\nabla u_t+a^{-1}\Delta u_t,
$$
we further divide
$J_1$
into three parts:
\begin{align*}
	J_1^2&\lesssim \int_{\mathbf{R}^n}a^{-1}|\Delta(a^{-1})u_t|^2dx
		+\int_{\mathbf{R}^n}a^{-1}|\nabla(a^{-1})|^2|\nabla u_t|^2dx
		+\int_{\mathbf{R}^n}a^{-1}|a^{-1}\Delta u_t|^2dx\\
	&\equiv J_{11}^2+J_{12}^2+J_{13}^2.
\end{align*}
By noting
$|\Delta (a^{-1})|\lesssim \langle x\rangle^{\alpha-2}$
and \eqref{ut}, we have
\begin{align*}
	J_{11}^2&\lesssim \int_{\mathbf{R}^n}\langle x\rangle^{4\alpha-4}
		a(x)u_t(\tau,x)^2dx\\
	&\lesssim (1+\tau)^{-\frac{n-\alpha}{2-\alpha}-2+\varepsilon}.
\end{align*}
Therefore, we immediately obtain
$$
	\int_0^{t/2}J_{11}d\tau\lesssim 1,
$$
provided that
$\varepsilon$
is sufficiently small.

Next, we estimate
$J_{12}$.
$$
	J_{12}^2\lesssim
		(1+\tau)^{\alpha}\int_{\mathbf{R}^n}|\nabla u_t(\tau,x)|^2dx
	\lesssim (1+\tau)^{\alpha-\frac{n-\alpha}{2-\alpha}-3+\varepsilon}
$$
and hence
$$
	\int_0^{t/2}J_{12}d\tau\lesssim 1.
$$
Since
$u_t$
also satisfies \eqref{DW}, we can rewrite
$$
	\Delta u_t=u_{ttt}-au_{tt}.
$$
Therefore, we have
\begin{align*}
	J_{13}^2&\lesssim \int_{\mathbf{R}^n}a(x)^{-4}a(x)|u_{ttt}(\tau,x)|^2dx
		+\int_{\mathbf{R}^n}a(x)^{-2}a(x)|u_{tt}(\tau,x)|^2dx\\
	&\lesssim (1+\tau)^{4\alpha-\frac{n-\alpha}{2-\alpha}-6+\varepsilon}
		+(1+\tau)^{2\alpha-\frac{n-\alpha}{2-\alpha}-4+\varepsilon}.
\end{align*}
These estimates imply
$$
	\int_0^{t/2}J_1d\tau\lesssim 1.
$$
This completes the proof.

\section*{Acknowledgement}
The author is deeply grateful to Professor Tatsuo Nishitani
for giving him constructive comments and warm
encouragement.

\end{document}